\documentclass[a4paper]{amsart}
\usepackage{graphicx}
\usepackage{amssymb}
\usepackage{amsmath}
\usepackage{amsthm,amsfonts,bbm}
\usepackage{amscd}
\usepackage[all,2cell]{xy}

\UseAllTwocells \SilentMatrices

\newtheorem{thm}{Theorem}[section]

\newtheorem{cor}[thm]{Corollary}
\newtheorem{lem}[thm]{Lemma}

\newtheorem{cons}[thm]{Construction}
\newtheorem{prop-def}[thm]{Proposition-Definition}
\theoremstyle{definition}
\newtheorem{defi}[thm]{Definition}
\theoremstyle{remark}
\newtheorem{rmk}[thm]{\bf Remark}
\newtheorem{exm}[thm]{\bf Example}
\numberwithin{equation}{section}
\numberwithin{figure}{section}

\def\A{\mathcal{A}}

\def\B{\mathcal{B}}
\def\C{\mathbb{C}}

\def \ep{\epsilon}

\def \HSpec{\mbox{\rm HSpec}}

\def\I{\mathcal{I}}

\def\la{\lambda}
\def \lamin{\lambda_{\min}}

\def \rank{\mbox{\rm rank}}

\def\T{\mathcal{T}}

\def\u{{\mathbf u}}
\def\v{{\mathbf v}}

\def\Z{\mathbb{Z}}

\begin{document}
\title[Non-odd-transversal hypergraph and non-odd-bipartite hypergraph]
{Minimal non-odd-transversal hypergraphs and minimal non-odd-bipartite hypergraphs}

\author[Y.-Z. Fan]{Yi-Zheng Fan$^*$}
\address{School of Mathematical Sciences, Anhui University, Hefei 230601, P. R. China}
\email{fanyz@ahu.edu.cn}
\thanks{$^*$The corresponding author.
This work was supported by National Natural Science Foundation of China (Grant No. 11871073, 11771016).}

\author[Y. Wang]{Yi Wang}
\address{School of Mathematical Sciences, Anhui University, Hefei 230601, P. R. China}
\email{wangy@ahu.edu.cn}

\author[J.-C. Wan]{Jiang-Chao Wan}
\address{School of Mathematical Sciences, Anhui University, Hefei 230601, P. R. China}
\email{wanjc@stu.ahu.edu.cn}

\date{\today}



\begin{abstract}
Among all uniform hypergraphs with even uniformity, the odd-transversal or
odd-bipartite hypergraphs are more close to bipartite simple graphs from the viewpoint of
both structure and spectrum.
A hypergraph is called minimal non-odd-transversal if it is non-odd-transversal but deleting any edge results in an odd-transversal hypergraph.
In this paper we give an equivalent characterization of the minimal non-odd-transversal hypergraphs by the degrees and the rank of its incidence matrix over $\Z_2$.
If a minimal non-odd-transversal hypergraph is uniform, then it has even uniformity, and hence is minimal non-odd-bipartite.
We characterize $2$-regular uniform  minimal non-odd-bipartite hypergraphs, and give some examples of $d$-regular uniform hypergraphs which are minimal non-odd-bipartite.
Finally we give upper bounds for the least H-eigenvalue of the adjacency tensor of minimal non-odd-bipartite hypergraphs.
\end{abstract}

\maketitle

\section{Introduction}
Let $G=(V,E)$ be a hypergraph, where $V=:V(G)$ is the vertex set, and $E=:E(G)$ is the edge set whose elements $e \subseteq V$.
If for each edge $e$ of $G$, $|e|=k$, then $G$ is called a \emph{$k$-uniform} hypergraph.
The \emph{degree} $d(v)$ of a vertex $v$ of $G$ is defined to be the number of edges of $G$ containing $v$.
If $d(v)=d$ for all vertices $v$ of $G$, then $G$ is called \emph{$d$-regular}.

A hypergraph $G$ is called \emph{$2$-colorable} if there exists a $2$-coloring of the vertices of $V(G)$ such that $G$ contains no monochromatic edges;
and it is called \emph{minimal non-$2$-colorable} if it is non-$2$-colorable but deleting any edge from $E(G)$ results in a $2$-colorable hypergraph.
Seymour \cite{Sey} proved that if $G$ is minimal non-$2$-colorable and $V(G)=\cup\{e \in E(G)\}$, then $|E(G)| \ge |V(G)|$.
Aharoni and Linial \cite{AhLi} presented an infinite version of Seymour's result.
Alon and Bregman \cite{AlBr} proved that if $k \ge 8$ then
every $k$-regular $k$-uniform hypergraph is $2$-colorable.
Henninga and  Yeoa \cite{HeYeEJC} showed that Alon-Bergman result is true for $k \ge 4$.

A subset $U$ of $V(G)$ is called a \emph{transversal} (also called \emph{vertex cover} or \emph{hitting set}) of $G$
  if each edge of $G$ has a nonempty intersection with $U$.
The \emph{transversal number} of $G$ is the minimum size of transversals in $G$,
  which was well studied by Alon \cite{Alon}, Chv\'{a}tal and McDiarmid \cite{ChMc}, Henninga and Yeo \cite{HeYeDM}.
$G$ is called \emph{bipartite} if for some nonempty proper subset $U \subseteq V(G)$, $U$ and its complement $U^c$ are both transversal;
or equivalently the vertex set $V(G)$ has a bipartition into two parts such that every edge of $E(G)$ intersects both parts.
Surely, $G$ is bipartite if and only if $G$ is 2-colorable.

A subset $U$ of $V(G)$ is called an \emph{odd transversal} of $G$ if each edge of $G$ intersects $U$ in an odd number of vertices \cite{CH, RS}.
A hypergraph $G$ is called \emph{odd-transversal} if it has an odd transversal.
Nikiforov \cite{Ni0} firstly uses odd transversal to investigate the spectral symmetry of tensors and hypergraphs.
Hu and Qi \cite{HQ} introduce the notion of odd-bipartite hypergraph to study the zero eigenvalue of the signless Laplacian tensor.

\begin{defi}[\cite{HQ}]
Let $G$ be a $k$-uniform hypergraph $G$, where $k$ is even.
If there exists a bipartition $\{U, U^c\}$ of $V(G)$ has such that each edge of $G$ intersects $U$ (and also $U^c$) in an odd number of vertices, then $G$ is called \emph{odd-bipartite}, and $\{U,U^c\}$ is an \emph{odd-bipartition} of $G$.
\end{defi}

So, odd-bipartite hypergraphs are surely odd-transversal hypergraphs and bipartite hypergraphs.
For the uniform hypergraphs with even uniformity,
the notion of odd-bipartite hypergraphs is \emph{equivalent to} that of odd-transversal hypergraphs.

From the viewpoint of spectrum, a simple graph is bipartite  if and only if its adjacency matrix has a symmetric spectrum.
However, the adjacency tensor of a bipartite uniform hypergraph does not possess such property.
We note that the hypergraphs under consideration are uniform when discussing their spectra.
Shao et al. \cite{SSW} proved that the adjacency tensor of a $k$-uniform hypergraph $G$ has a symmetric H-spectrum if and only if $k$ is even and $G$ is odd-bipartite.
So, the odd-bipartite hypergraphs are more close to bipartite simple graphs than the bipartite hypergraphs based on the following two reasons.
First they both have a structural property, namely, there exists a bipartition of the vertex set such that
  every edge intersects the each part of the bipartition in an odd number of vertices.
Second they both have a symmetric H-spectrum.

There are some examples of odd-bipartite hypergraphs, e.g. power of simple graphs and cored hypergraphs \cite{HQS}, hm-hypergraphs \cite{HQ},
$m$-partite $m$-uniform hypergraphs \cite{CD}.
Nikiforov \cite{Ni} gives two classes of non-odd-transversal hypergraphs.
Fan et al. \cite{KF} construct non-odd-bipartite generalized power hypergraphs from non-bipartite simple graphs.
It is known that a connected bipartite simple graph has a unique bipartition up to isomorphism.
However, an odd-bipartite hypergraph can have more than one odd-bipartition.
Fan et al. \cite{FY} given a explicit formula for the number of odd-bipartition of a hypergraph by the rank of its incidence matrix over $\Z_2$.
So, it seems hard to give examples of non-odd-bipartite hypergraphs.

To our knowledge, there is no characterization of non-odd-transversal or non-odd-bipartite hypergraphs.
We observe that non-odd-transversal hypergraphs have a hereditary property, that is,
if $G$ contains a  non-odd-transversal sub-hypergraph, then $G$ is non-odd-transversal.
$G$ is called \emph{minimal non-odd-transversal},
if $G$ is non-odd-transversal but deleting any edge from $G$ results in an odd-transversal hypergraph,
or equivalently, any nonempty proper edge-induced sub-hypergraph of $G$ is odd-transversal.
In this paper we give an equivalent characterization of the minimal non-odd-transversal hypergraphs by the degrees and the rank of its incidence matrix over $\Z_2$.
If a minimal non-odd-transversal hypergraph is uniform, then it has  even uniformity, and hence is minimal non-odd-bipartite.
We characterized $2$-regular uniform  minimal non-odd-bipartite hypergraphs, and give some examples of $d$-regular uniform hypergraphs which are minimal non-odd-bipartite.
Finally we give upper bounds for the least H-eigenvalue of the adjacency tensor of minimal non-odd-bipartite hypergraphs.

\section{Basic notions}
Unless specified somewhere, all hypergraphs in this paper contain no multiple edges or isolated vertices,
where vertex is called \emph{isolated } if it is not contained in any edge of the hypergraph.
Let $G=(V,E)$ be a hypergraph.
$G$ is called \emph{square} if $|V|=|E|$.
A \emph{walk} of length $t$ in $G$ is a sequence of alternate vertices and edges: $v_{0}e_{1}v_{1}e_{2}\ldots e_{t}v_{t}$,
    where $\{v_{i},v_{i+1}\}\subseteq e_{i}$ for $i=0,1,\ldots,t-1$.
$G$ is said to be \emph{connected} if every two vertices are connected by a walk.

The \emph{vertex-induced sub-hypergraph} of $G$ by the a subset $U \subseteq V(G)$, denoted by $G|_U$,
  is a hypergraph with vertex set $U$ and edge set $\{e \cap U: e \in E(G), e \cap U \ne \emptyset\}$.
For a connected hypergraph $G$, a vertex $v$ is called a \emph{cut vertex} of $G$ if $G|_{V(G)\backslash \{v\}}$ is disconnected.
The \emph{edge-induced sub-hypergraph} of $G$ by a subset $F \subseteq E(G)$, denoted by $G|_F$, is a hypergraph with vertex set $\cup_{e \in F}e$ and edge set $F$.

Let $G$ be a hypergraph and let $e$ be an edge of $G$.
Denote by $G-e$ the hypergraph obtained from $G$ by deleting the edge $e$ from $E(G)$.
For a connected hypergraph $G$, an edge $e$ is called a \emph{cut edge} of $G$ if $G-e$ is disconnected.

A \emph{matching} $M$ of $G$ is a set of pairwise disjoint edges of $G$.
In particular, if $G$ is bipartite simple graph with a bipartition $\{V_1,V_2\}$,
a vertex subset $U_1 \subseteq V_1$ is \emph{matched to} $U_2 \subseteq V_2$ in $M$,
if there exists a bijection $f: U_1 \to U_2$ such that $\{\{v,f(v)\}: v \in U_1\} \subseteq M$.
A subset $U$ of $V_1$ (or $V_2$) is \emph{matched by} $M$ if every vertex of $U$ is incident with an edge of $M$.
$M$ is called a \emph{perfect matching} if $V_1$ and $V_2$ are both matched by $M$.

The \emph{incidence bipartite graph} $\Gamma_G$ of $G$ is a bipartite simple graph with two parts $V(G)$ and $E(G)$ such that
$\{v, e\} \in E(\Gamma_G)$ if and only if $v \in e$.

The edge-vertex \emph{incidence matrix} of $G$, denoted by $B_G=(b_{e,v})$, is a matrix of size $|E(G)| \times |V(G)|$,
whose entries $b_{e,v}=1$ if $v \in e$, and $b_{e,v}=0$ otherwise.

The \emph{dual} of $G$, denoted by $G^\ast$, is the hypergraph whose vertex set is $E(G)$ and edge set is $\{\{e \in E(G): v \in e\}: v \in V(G)\}$.
If no two vertices of $G$ are contained in precisely the same edges of $G$, then $(G^\ast)^\ast$ is isomorphic to $G$.
In this situation, the incidence bipartite graph $\Gamma_G \cong \Gamma_{G^\ast}$, and the incidence matrix $B_G=B^\top_{G^\ast}$,
where the latter denotes the transpose of $B_{G^\ast}$.

Let $G$ be a simple graph, and let $k$ be even integer greater than $2$.
Denote by $G^{k, {k \over 2}}$ the hypergraph obtained from $G$ whose vertex set is $\cup_{v \in V(G)} \v$ and edge set
$\{ \u \cup \v: \{u,v\} \in E(G)\}$, where $\v$ denotes an ${k \over 2}$-set corresponding to $v$, and all those sets are pairwise disjoint;
intuitively  $G^{k, {k \over 2}}$ is obtained from $G$ by blowing up each vertex into an ${k \over 2}$-set and preserving the adjacency relation \cite{KF}.
It is proved that $G^{k, {k \over 2}}$ is non-odd-bipartite if and only if $G$ is non-bipartite \cite{KF}.

Next we will introduce some knowledge of eigenvalues of a tensor.
For integers $k\geq 2$ and $n\geq 2$,
  a \emph{tensor} (also called \emph{hypermatrix}) $\mathcal{T}=(t_{i_{1}\ldots i_{k}})$ of order $k$ and dimension $n$ refers to a
  multidimensional array $t_{i_{1}i_2\ldots i_{k}}$ such that $t_{i_{1}i_2\ldots i_{k}}\in \mathbb{C}$ for all $i_{j}\in [n]:=\{1,2,\ldots,n\}$ and $j\in [k]$.
$\mathcal{T}$ is called \textit{symmetric} if its entries are invariant under any permutation of their indices.

 Given a vector $x\in \mathbb{C}^{n}$, $\mathcal{T}x^{k} \in \mathbb{C}$, and $\mathcal{T}x^{k-1} \in \mathbb{C}^{n}$, which are defined as follows:
   $$\mathcal{T}x^{k}=\sum_{i_1,i_{2},\ldots,i_{k}\in [n]}t_{i_1i_{2}\ldots i_{k}}x_{i_1}x_{i_{2}}\cdots x_{i_k},$$
   $$   (\mathcal{T}x^{k-1})_i=\sum_{i_{2},\ldots,i_{k}\in [n]}t_{ii_{2}i_3\ldots i_{k}}x_{i_{2}}x_{i_3}\cdots x_{i_k}, \mbox{~for~} i \in [n].$$
 Let $\mathcal{I}$ be the {\it identity tensor} of order $k$ and dimension $n$, that is, $i_{i_{1}i_2 \ldots i_{k}}=1$ if and only if
   $i_{1}=i_2=\cdots=i_{k} \in [n]$ and zero otherwise.

\begin{defi}{\em\cite{Lim,Qi}} Let $\mathcal{T}$ be a $k$-th order $n$-dimensional real tensor.
For some $\lambda \in \mathbb{C}$, if the polynomial system $(\lambda \mathcal{I}-\mathcal{T})x^{k-1}=0$, or equivalently $\mathcal{T}x^{k-1}=\lambda x^{[k-1]}$, has a solution $x\in \mathbb{C}^{n}\backslash \{0\}$,
then $\lambda $ is called an eigenvalue of $\mathcal{T}$ and $x$ is an eigenvector of $\mathcal{T}$ associated with $\lambda$,
where $x^{[k-1]}:=(x_1^{k-1}, x_2^{k-1},\ldots,x_n^{k-1}) \in \mathbb{C}^n$.
\end{defi}

The \emph{characteristic polynomial} $\varphi_\T(\la)$ of $\T$ is defined as the resultant of the polynomials $(\la \I-\T)x^{k-1}$; see \cite{Qi,CPZ2,Ha}.
It is known that $\la$ is an eigenvalue of $\T$ if and only if it is a root of $\varphi_\T(\la)$.
The \emph{spectrum} of $\T$ is the multi-set of the roots of $\varphi_\T(\la)$.

Suppose that $\T$ is real.
If $x$ is a real eigenvector of $\mathcal{T}$, surely the corresponding eigenvalue $\lambda$ is real.
In this case, $x$ is called an {\it $H$-eigenvector} and $\lambda$ is called an {\it $H$-eigenvalue}.
The \emph{H-spectrum} of $\T$ is the set of all H-eigenvalues of $\T$, denote by $\HSpec(\T)$.
The \emph{spectral radius} of $\T$ is defined as the maximum modulus of the eigenvalues of $\T$, denoted by $\rho(\T)$.
%$$\rho(\T)=\max\{|\lambda|: \lambda \mbox{ is an eigenvalue of } \T \}.$$
Denote by $\lambda_{\max}(\T), \lambda_{\min}(\T)$ the largest H-eigenvalue and the least H-eigenvalue of $\T$, respectively.

For a symmetric tensor, we have the following result.

\begin{lem} \label{semi}
Let $\T$ be a real symmetric tensor of order $k$ and dimension $n$.
Then

\begin{enumerate}

\item
{\em [\cite{ZhQiWu}, Theorem 3.6]} If $\T$ is also nonnegative, then
$$ \lambda_{\max}(\T)=\min\{\T x^k: x\in \mathbb{R}^{n}, x \ge 0,  \|x\|_k=1\},$$
where $\|x\|_k=\left(\sum_{i=1}^{n} |x_i^k|\right)^{1 \over k}$.
Furthermore, $x$ is an optimal solution of the above optimization
if and only if it is an eigenvector of $\T$ associated with $\lambda_{\max}(\T)$.

\item
{\em [\cite{Qi}, Theorem 5]} If $k$ is also even, then
$$\lambda_{\min}(\T)=\min\{\T x^k: x\in \mathbb{R}^{n}, \|x\|_k=1\},$$
and $x$ is an optimal solution of the above optimization
if and only if it is an eigenvector of $\T$ associated with $\lambda_{\min}(\T)$.
\end{enumerate}
\end{lem}

Let $G$ be a $k$-uniform hypergraph on $n$ vertices $v_1,v_2,\ldots,v_n$.
The {\it adjacency tensor} of $G$ \cite{CD} is defined as $\mathcal{A}(G)=(a_{i_{1}i_{2}\ldots i_{k}})$, an order $k$ dimensional $n$ tensor, where
\[a_{i_{1}i_{2}\ldots i_{k}}=\left\{
 \begin{array}{ll}
\frac{1}{(k-1)!}, &  \mbox{if~} \{v_{i_{1}},v_{i_{2}},\ldots,v_{i_{k}}\} \in E(G);\\
  0, & \mbox{otherwise}.
  \end{array}\right.
\]
The \emph{spectral radius, the least H-eigenvalue} of $G$ are referring to its adjacency tensor $\A(G)$, denoted by $\rho(G), \lamin(G)$ respectively.
The H-spectrum of $\A(G)$ is denoted by $\HSpec(G)$.

The spectral hypergraph theory has been an active topic in algebraic graph theory recently; see e.g. \cite{CD,FBH, FHB,Ni,PZ}.
By the Perron-Frobenius theorem for nonnegative tensors \cite{CPZ, FGH, YY1,YY2,YY3}, $\rho(G)$ is exactly the largest H-eigenvalue of $\A(G)$.
If $G$ is connected, there exists a unique positive eigenvector up to scales associated with $\rho(G)$, called the {\it Perron vector} of $G$.
Noting that the adjacency tensor $\A(G)$ is nonnegative and symmetric, so $\rho(G)$ holds (1) of Lemma \ref{semi}, and $\lamin(G)$ holds (2) of Lemma \ref{semi} if $k$ is even.
By Perron-Frobenius theorem, $\lamin(G) \ge -\rho(G)$.
By the following lemma, if $G$ is connected and non-odd-bipartite, then $\lamin(G) > -\rho(G)$.

\begin{lem}\cite{Ni,Ni0, SSW, FY}\label{ob-equiv}
Let $G$ be a $k$-uniform connected hypergraph. Then the following results are equivalent.

\begin{enumerate}
\item $k$ is even and $G$ is odd-bipartite.

\item $\lamin(G) = -\rho(G)$.

\item $\HSpec(G)=-\HSpec(G)$.

\end{enumerate}
\end{lem}

Finally, we introduce some notations used throughout out the paper.
Denote by $C_n$ a cycle of length $n$ as a simple graph.
Denote by $\mathbbm{1}$ an all-one vector whose size can be implicated by the context,
$\rank A$ the rank of a matrix $A$ over $\Z_2$, and $\mathbb{F}_q$ a field of order $q$.

\section{Characterization of minimal non-odd-transversal hypergraphs}

In this section we will give some equivalent conditions in terms of degrees and rank of the incidence matrix over $\Z_2$ for a hypergraph to be minimal non-odd-transversal.

\begin{lem}\label{conn-cut}
If $G$ is a minimal non-odd-transversal hypergraph, then $G$ is connected and contains no cut vertices.
\end{lem}

\begin{proof}
If $G$ contains more than one connected component, then at least one of them is non-odd-transversal,  a contradiction to the definition.
So $G$ itself is connected.
Suppose $G$ contains a cut vertex.
Then $G$ is obtained from two connected nontrivial sub-hypergraphs $G_1,G_2$ sharing exactly one vertex (the cut vertex).
So, at least one of $G_1,G_2$ is non-odd-transversal, also a contradiction.
\end{proof}

\begin{lem}\label{eqn-rank}
Let $G$ be a connected hypergraph, and $B_G$ be the edge-vertex incidence matrix of $G$.
Then $G$ is odd-transversal if and only if the equation
\begin{equation}\label{inc}
B_G x=\mathbbm{1} \mbox{~over~} \Z_2
\end{equation}
has a solution, or equivalently
\begin{equation}\label{rank}
\rank B_G =\rank (B_G, \mathbbm{1}) \mbox{~over~} \Z_2.
\end{equation}
\end{lem}

\begin{proof}
If $G$ is odd-transversal, then there is an odd-transversal $U$ of $G$.
Define a vector $ x \in \Z_2^{V(G)}$ such that $x_v=1$ if $v \in U$, and $x_v =0$ otherwise.
By the definition, it is easy to verify that $x$ is a solution of the equation (\ref{inc}).
On the other hand, if $x$ is a solution of the equation (\ref{inc}), define $U=\{v: x_v =1\}$.
Then $U \ne \emptyset$, and for each edge $e$ of $G$, $|e \cap U|$ is odd, implying that $G$ is odd-transversal.
\end{proof}

For each edge $e \in E(G)$, define an \emph{indicator vector} $\chi_e \in \Z_2^{V(G)}$ such that $\chi_e(v)=1$ if $v \in e$ and $\chi_e(v)=0$ otherwise.
Then $B_G$ consists of those $\chi_e$ as row vectors for all $e \in E(G)$.

\begin{lem}\label{odd-sum}
Let $G$ be a connected hypergraph with $m$ edges.
If $m$ is odd, and each vertex has an even degree, or equivalently $\sum_{e \in E(G)} \chi_e=0$ over $\Z_2$, then $G$ is non-odd-transversal.
\end{lem}

\begin{proof}
Let $e_1,\ldots,e_m$ be edges of $G$.
Write $(B_G,\mathbbm{1})$ as the following form:
\begin{equation}\label{bg1}
(B_G,\mathbbm{1})=\left(
  \begin{array}{cc}
    \chi_{e_1} & 1 \\
    \chi_{e_2} & 1 \\
    \vdots & \vdots \\
    \chi_{e_m} & 1 \\
  \end{array}
\right).
\end{equation}
Adding the first row to all other rows over $\Z_2$, we will have
\begin{equation}\label{bg1T}
\left(
  \begin{array}{cc}
    \chi_{e_1} & 1 \\
    \chi_{e_2}+\chi_{e_1} & 0 \\
    \vdots & \vdots \\
    \chi_{e_m}+\chi_{e_1} & 0 \\
  \end{array}
\right)=:\left(
           \begin{array}{cc}
             \chi_{e_1} & 1 \\
             C & O \\
           \end{array}
         \right).
\end{equation}
So, $\rank(B_G,\mathbbm{1})=1+\rank C$.
As $m$ is odd and $\sum_{e \in E(G)} \chi_e=0$, $$\chi_{e_1}=\sum_{i=2}^m (\chi_{e_i}+\chi_{e_1}),$$
 implying that $\rank B_G =\rank C$.
By Lemma \ref{eqn-rank}, $G$ is non-odd-transversal.
\end{proof}

\begin{thm}\label{main}
Let $G$ be a connected hypergraph with $m$ edges. The following are equivalent.
\begin{enumerate}
  \item $G$ is minimal non-odd-transversal.
  \item $m$ is odd, $\sum_{e \in E(G)}\chi_e=0$ over $\Z_2$, and $\sum_{e \in F} \chi_e \ne 0$ over $\Z_2$ for any nonempty proper subset $F$ of $E(G)$.
  \item $m$ is odd, $\sum_{e \in E(G)}\chi_e=0$ over $\Z_2$, and $\rank B_G=m-1$ over $\Z_2$.

  \item $m$ is odd, each vertex of $G$ has an even degree, and any nonempty proper edge-induced sub-hypergraph of $G$ contains vertices of odd degrees.
\end{enumerate}
\end{thm}

\begin{proof}
$(1) \Rightarrow (2)$.
Suppose that $G$ is minimal non-odd-transversal.
By Eq. (\ref{bg1}) and Eq. (\ref{bg1T}), as $\rank B_G \ne \rank(B_G,\mathbbm{1})$ over $\Z_2$ by Lemma \ref{eqn-rank},
$\chi_{e_1}$ is a linear combination of $\chi_{e_i}+\chi_{e_1}$ for $i=2,\ldots,m$.
So there exist $a_i \in \Z_2$ for $i=2, \ldots, m$ such that
\begin{equation}\label{lincom} \chi_{e_1}=\sum_{i=2}^m a_i(\chi_{e_i}+\chi_{e_1})=\left(\sum_{i=2}^m a_i \right)e_1+\sum_{i=2}^m a_i\chi_{e_i}.\end{equation}
We assert that $a_i=1$ for $i=2,\ldots,m$.
Otherwise, there exists a $j$, $2 \le j \le m$, such that $a_j=0$.
Then $\chi_{e_1}$ is also a linear combination of $\chi_{e_i}+\chi_{e_1}$ for $i=2,\ldots,m$ and $i \ne j$.
So, $\rank B_{G-e_j} \ne \rank(B_{G-e_j},\mathbbm{1})$, implying that $G-e_j$ is non-odd-transversal by Lemma \ref{eqn-rank}, a contradiction to the definition.

If $m$ is even, then $\sum_{i=2}^m \chi_{e_i}=0$ by Eq. (\ref{lincom}), implying the vertices of $V(G)$ all have even degrees in $G-e_1$.
So the vertices of $e_1$ all have odd degrees in $G$.
By the arbitrariness of $e_1$, each vertex has an odd degree in $G$.
However, there exists a vertex $v \notin e_1$ so that $d_G(v)=d_{G-e_1}(v)$, which is an even number, a contradiction.

So, $m$ is odd, and $\sum_{e \in E(G)}\chi_e=0$ by Eq. (\ref{lincom}).
Assume to the contrary there exists a nonempty proper subset $F$ of $E(G)$, $\sum_{e \in F} \chi_e= 0$ over $\Z_2$.
If $|F|$ is odd, then by Lemma \ref{odd-sum}, the sub-hypergraph $G|_F$ induced by the edges of $F$ is non-odd-transversal, a contradiction to the definition.
Otherwise, $|F|$ is even, then $|E(G)\backslash F|$ is odd as $m$ is odd, and the sub-hypergraph $G|_{E(G)\backslash F}$ is non-odd-transversal, also a contradiction.

$(2) \Rightarrow (3)$.
As $\sum_{e \in E(G)}\chi_e=0$,  $\rank B_G \le m-1$ over $\Z_2$.
If $\rank B_G \le m-2$ over $\Z_2$, then $\chi_{e_1}, \ldots, \chi_{e_{m-1}}$ are linear dependent.
So there exists $a_1,\ldots, a_{m-1} \in \Z_2$, not all being zero, such that $\sum_{i=1}^{m-1} a_i\chi_{e_i}=0$.
Taking $F=\{e_i: a_i=1, 1 \le i \le m-1\}$, then $\sum_{e \in F} \chi_{e}=\sum_{i=1}^{m-1} a_i\chi_{e_i}=0$, a contradiction to (2).
So $\rank B_G =m-1$ over $\Z_2$.

$(3) \Rightarrow (1)$. By Lemma \ref{odd-sum}, $G$ is non-odd-transversal.
Let $e$ be an arbitrary edge of $G$.
Adding all rows $\chi_f$ for $f \ne e$ to the row $\chi_e$ will yield a zero row as $\sum_{e \in E(G)}\chi_e=0$.
%As $\sum_{e \in E(G)}\chi_e=0$, $\chi_{\bar{e}} =\sum_{e \in E(G-\bar{e})}\chi_e$.
So $\rank B_{G-e}=\rank B_G=m-1$ over $\Z_2$,
implying that $B_{G-e}$ has full rank over $\Z_2$ with respect to rows.
Hence,  $\rank B_{G-e}=\rank (B_{G-e},\mathbbm{1})$ over $\Z_2$, and $G-e$ is odd-transversal by Lemma \ref{eqn-rank}.
So $G$ is minimal non-odd-transversal.

Of course (2) is equivalent to (4).
\end{proof}

\begin{rmk}\label{m-1}
From the proof of $(3) \Rightarrow (1)$ in Theorem \ref{main},
if $G$ is minimal non-odd-transversal hypergraphs with $m$ edges, then any $m-1$ rows of $B_G$ are linear independent over $\Z_2$.
\end{rmk}

\begin{exm}\label{genexm} The following are minimal non-odd-transversal hypergraphs by verifying the degrees and
 the rank of incidence matrix over $\Z_2$ according to Theorem \ref{main}, where the last two hypergraphs are square.
\begin{enumerate}
\item $\{1,2,3,4\}$, $\{2,3,4,5\}$, $\{1,5\}$.

\item $\{1,2,3\}$, $\{2,3,4\}$, $\{3,4,5\}$, $\{1,4,5\}$, $\{3,4\}$.

\item $\{1,2,3\}$, $\{1,3,4,5\}$, $\{1,2,4,6\}$, $\{1,5,6,7\}$, $\{2,4,7\}$, $\{2,5,6,7\}$, $\{4,5,6,7\}$.
\end{enumerate}
\end{exm}

From Example \ref{genexm}, we know a minimal non-odd-transversal hypergraph can contain both even-sized edges and odd-sized edges.
In the following we will discuss minimal non-odd-transversal hypergraphs only with even-sized edges.

\begin{cor}\label{nm}
Let $G$ be an minimal non-odd-transversal hypergraph only with even-sized edges, which has $n$ vertices and $m$ edges.
Then the following results hold.

\begin{enumerate}
  \item $n \ge m$.
  \item For $1 \le t \le m-1$, any $t$ edges intersect at least $t+1$ vertices.
  \item The incidence bipartite graph $\Gamma_G$ has a matching $M$ such that $E(G)$ is matched by $M$, namely,
  there exists an injection $f: E(G) \to V(G)$ such that $f(e) \in e$ for each $e \in E(G)$.
\end{enumerate}

\end{cor}

\begin{proof}
Consider the incidence matrix $B_G$ of $G$.
As $G$ contains only even-sized edges, each row sum of $B_G$ is zero over $\Z_2$, which implies $\rank B_G \le n-1$.
By Theorem \ref{main}(3), $\rank B_G=m-1$, yielding the result (1).

Let $e_1,\ldots, e_t$ be $t$ edges of $G$, where $1 \le t \le m-1$.
Let $U=\cup_{i=1}^t e_i$.
Let $B_G[e_1,\ldots,e_t|U]$ be the sub-matrix of $B_G$ with rows indexed $e_1,\ldots, e_t$ and columns indexed by the vertices of $U$.
By Remark \ref{m-1}, $\rank B_G[e_1,\ldots,e_t|U]=t \le |U|-1$, as each row sum of the sub-matrix is zero.
So we have $|U| \ge t+1$, yielding the result (2).

The result (3) follows from Hall's Theorem.
\end{proof}

\begin{cor}\label{pm}
Let $G$ be a square minimal non-odd-transversal hypergraph only with even-sized edges.
Then

\begin{enumerate}
  \item The incidence bipartite graph $\Gamma_G$ has a perfect matching, namely, there exists a bijection $f: E(G) \to V(G)$ such that $f(e) \in e$ for each $e \in E(G)$.

  \item For each nonempty proper subset $U$ of $V(G)$, $G|_U$ contains at least $|U|+1$ edges, and also contains odd-sized edges.

\end{enumerate}
\end{cor}

\begin{proof}
Surely (1) comes from (3) of Corollary \ref{nm} as $G$ is square.
Now let $U$ be a nonempty proper subset of $V(G)$.
Let $F$ be the set of edges that intersect $U$ so that $G|_U$ has edges $e \cap U$ for all $e \in F$.
If $F=E(G)$, then $|F|=|V(G)| \ge |U|+1$ as $G$ is square.
Otherwise, we consider the submatrix $B_G[F^c|U^c]$, which has rank $|F^c|$ from the its rows by  Remark \ref{m-1}.
So, $|F^c|=n-|F| \le n-|U|-1$ as each row sum of $B_G[F^c|U^c]$ is zero over $\Z_2$, implying that $|F| \ge |U|+1$.

Assume to the contrary that each edge of $G|_U$ has even size.
 Then $B_G[E(G)|U]$, and $B_G[E(G)|U^c]$ as well, has zero row sums.
So
\begin{eqnarray*}
\rank B_G & \le & \rank B_G[E(G)|U]+ \rank B_G[E(G)|U^c] \\
& \le & |U|-1 + |U^c|-1 = |V(G)|-2=|E(G)|-2,
\end{eqnarray*}
 a contradiction to Theorem \ref{main}(3).
\end{proof}

\begin{cor}\label{pm2}
Let $G$ be a square hypergraph only with even-sized edges and even-degree vertices.
Then $G$ is minimal non-odd-transversal if and only if its dual $G^\ast$ is minimal non-odd-transversal.
\end{cor}

\begin{proof}
Suppose $G$ is minimal non-odd-transversal with $n$ vertices (edges).
By Corollary \ref{pm}(2), no two vertices of $G$ lie in precisely the same edges of $G$.
So $G^*$ is also square, and $B_{G^\ast}=B^\top_G$.
As each edge of $G$ is even sized, each vertex of $G^*$ has even degree.
So $G^*$ is minimal non-odd-transversal by Theorem \ref{main}.
As $G$ is isomorphic to $(G^\ast)^\ast$, $G$ is minimal non-odd-transversal if $G^\ast$ is.
\end{proof}

\section{minimal non-odd-bipartite regular hypergraphs}
In this section we mainly discuss minimal non-odd-transversal $k$-uniform hypergraphs $G$.
By the following lemma,  $k$ is necessarily even.
So the minimal non-odd-transversal uniform hypergraphs are \emph{exactly} the minimal non-odd-bipartite hypergraphs.

\begin{lem}\label{prep}
Let $G$ be a minimal non-odd-transversal $k$-uniform hypergraphs $G$, which has  $n$ vertices and $m$ edges.
Then $k$ is even.
If $G$ is further $d$-regular, then $d$ is even and $d \le k$.
\end{lem}

\begin{proof}
By Theorem \ref{main}(4), each vertex of $G$ has an even degree so that the sum of degrees of the vertices of $G$ is even, which is equal to $mk$.
As $m$ is odd by Theorem \ref{main}, $k$ is necessarily even, which implies that $n \ge m$ by Corollary \ref{nm}(1).
If $G$ is $d$-regular, $d$ is even by Theorem \ref{main}(4).
Surely  $nd=mk$,  so $d \le k$ as $n \ge m$.
\end{proof}

So, in the following discussion we only deal with non-odd-bipartite regular hypergraphs with even uniformity and even degree.

\subsection{$2$-regular minimal non-odd-bipartite hypergraphs}
It is known that the only minimal non-bipartite simple graph is an odd cycle $C_{2l+1}$, which is $2$-regular.
As a simple generalization, the generalized power hypergraph $C_{2l+1}^{k, {k \over 2}}$ is a $2$-regular minimal non-odd-bipartite $k$-uniform hypergraph.
However, the above hypergraph is not the only $2$-regular minimal non-odd-bipartite $k$-uniform hypergraph.
For example, the following $4$-uniform hypergraph on $10$ vertices with $5$ edges below is minimal non-odd-bipartite:
$$ \{1,2,3,4\}, \{3,4,5,6\}, \{5,6,7,8\}, \{1,7,9,10\}, \{2,8,9,10\}.$$

\begin{lem}\label{2-con}
Let $G$ be a $2$-regular uniform hypergraph with an odd number of edges.
If $G$ is connected, then $G$ is minimal non-odd-bipartite.
\end{lem}

\begin{proof}
Let $H$ be a nonempty proper edge-induced sub-hypergraph of $G$.
As $G$ is connected, $H$ contains a vertex $v$, which is also contained in some edge not in $H$.
So $v$ has degree $1$ in $H$.
The result follows by Theorem \ref{main}(4).
\end{proof}

Next we give a construction of $2$-regular $k$-uniform hypergraphs, where $k$ is an even integer greater than $2$.

\begin{cons}\label{cons-det}
Let $k$ be an even integer greater than $2$, and let $n,m$ be positive integers such that $n=\frac{km}{2}$.
Let $K_{n,n}$ be a complete bipartite simple graph with two parts $U_1$ and $U_2$, where $U_1=[n]$ and $U_2=\cup_{t=1}^m\{e_t^1,\ldots,e_t^{k \over 2}\}$.
Let $\hat{K}_{n,n}$ be obtained from $K_{n,n}$ by by deleting the edges between the vertices of $V_t:=\{{k \over 2}(t-1)+1,\ldots,{k \over 2}t\}$
 and the vertices of $E_t:=\{e_t^1, \ldots, e_t^{k \over 2}\}$ for $t \in [m]$.

Let $\hat{M}$ be a perfect matching of $\hat{K}_{n,n}$ such that $W_t:=\{i_{t1},\ldots,i_{t,{k \over 2}}\}$ are matched to $E_t$ respectively for $t \in [m]$,
and if $W_{t}=V_{s}$ for some $s \ne t$, then $W_{s}  \ne V_{t}$.

Define a hypergraph $G$ with vertex set $[n]$, whose edges are
\begin{equation}\label{edgeconst-con}
e_t=V_{t} \cup W_{t}, \hbox{~for~} t \in [m].
\end{equation}
 \end{cons}

\begin{lem}\label{cons-pro}
The hypergraph $G$ defined in Construct \ref{cons-det} is a $2$-regular $k$-uniform hypergraph on $n$ vertices.
\end{lem}

\begin{proof}
As there is no edge between $V_{t}$ and $E_t$ in $\hat{K}_{n,n}$, $W_{t} \cap V_{t}=\emptyset$ for each $t \in [m]$,
So each edge $e_t$ contains exactly $k$ vertices.
Note that $\{V_1,\ldots, V_t\}$ and $\{W_1,\ldots, W_t\}$ both form a $t$-partition of $[n]$.
For each vertex $v$ of $G$, $v \in V_s$ for a unique $s \in [m]$ and $v \in W_t$ for a unique $t \in [m]$, where $t \ne s$ as $V_s \cap W_s =\emptyset$.
So $v$ contained in exactly two edges $e_s$ and $e_t$, implying $v$ has degree $2$.
Finally we note that $G$ contains no multiple edges;
otherwise, if $e_s=e_t$ for $s \ne t$, then $V_s \cup W_s=V_t \cup W_t$,
which implies that $V_s=W_t$ and $V_t=W_s$ as $V_s \cup V_t =\emptyset$ and $W_s \cup W_t =\emptyset$, a contradiction to the assumption.
The result follows.
\end{proof}

\begin{cor}\label{consmth}
Any $2$-regular $k$-uniform hypergraph on $n$ vertices can be constructed as in Construction \ref{cons-det}, where $k$ is even integer greater than $2$.
\end{cor}

\begin{proof}
Let $G$ be a $2$-regular $k$-uniform hypergraph with $V(G)=[n]$ and $E(G)=\{e_1,\ldots,e_m\}$.
Surely, $n=\frac{km}{2}$.
Let $\bar{G}:=\frac{k}{2} \cdot G $ be a $k$-uniform hypergraph with vertex set $V(G)$ and edge set $\frac{k}{2} \cdot E(G):=\{\frac{k}{2} \cdot e: e \in E(G)\}$,
 where ${k \over 2} \cdot e$ means the $k \over 2$ copies of $e$, written as $e^1,\ldots,e^{k \over 2}$.
Then  $\bar{G}$ is a $k$-uniform $k$-regular multi-hypergraph on $n$ vertices.
The incidence bipartite graph $\Gamma_{\bar{G}}$ of $\bar{G}$ is $k$-regular.

Let $K_{\bar{G}}$ be a complete bipartite graph with two parts $V(\bar{G})$ and $E(\bar{G})$.
Let $\hat{K}_{\bar{G}}$ be obtained from $K_{\bar{G}}$ by deleting the edges
 between the vertices of $V_t:=\{{k \over 2}(t-1)+1,\ldots,{k \over 2}t\}$
 and the vertices of $E_t:=\{e_t^1, \ldots, e_t^{k \over 2}\}$ for $t \in [m]$.
Then $\hat{K}_{\bar{G}}$ is an $(n-{k \over 2})$-regular bipartite graph.

Considering the $k$-regular bipartite graph $\Gamma_{\bar{G}}$, it contains a perfect matching $M$.
By a possible relabeling of the vertices, we may assume that
for $t \in [m]$,
$V_{t}:=\{{k \over 2}(t-1)+1,\ldots,{k \over 2}t\}$ is matched to $E_t:=\{e_t^1,\ldots,e_t^{k \over 2}\}$ in $M$.
By the construction of $\bar{G}$, returning to $G$, $V_{t} \subseteq e_t$ for $t \in [m]$.

Now deleting the edges between the vertices of $V_{t}$ and the vertices of $E_t$ from $\Gamma_{\bar{G}}$ for $t \in [m]$,
we arrive at a ${k \over 2}$-regular bipartite graph denoted by $\hat{\Gamma}_{\bar{G}}$, which is a subgraph of $\hat{K}_{\bar{G}}$.
Now $\hat{\Gamma}_{\bar{G}}$, and hence $\hat{K}_{\bar{G}}$ has a perfect matching $\hat{M}$, where, for $t \in [m]$,
$W_{t}:=\{i_{t1},\ldots,i_{t,{k \over 2}}\}$ is matched to $E_t$ in $\hat{M}$.
So, returning to $G$, $W_{t} \subseteq e_t$ for $t \in [m]$.
As there is no edge between $V_{t}$ and $E_t$ in $\hat{\Gamma}_{\bar{G}}$, $W_{t} \cap V_{t}=\emptyset$ for each $t \in [m]$,
which implies that $e_t =V_{t} \cup W_{t}$ for $ t \in [m]$.

As $G$ contains no multiple edges, if $W_{t}=V_{s}$ for some $s \ne t$,
surely $W_{s}  \ne V_{t}$; otherwise $e_t=e_s=V_t \cup V_{s}$, a contradiction.

From the above discussion, $K_{\bar{G}}$ and $\hat{K}_{\bar{G}}$ are respectively isomorphic to $K_{n,n}$ and $\hat{K}_{n,n}$.
A perfect matching $\hat{M}$ in $\hat{K}_{\bar{G}}$ is isomorphic to a perfect matching in $\hat{K}_{n,n}$.
So $G$ can be constructed as in Construction \ref{cons-det}.
\end{proof}

\begin{thm}
Let $G$ be a $2$-regular $k$-uniform hypergraphs with $n$ vertices and $m$ edges, where $m$ is odd and $k$ is even.
Then $G$ is a minimal non-odd-bipartite hypergraph if and only if $G$ can be constructed as in Construction \ref{cons-det} and $G$ is connected.
\end{thm}

\begin{proof}
The sufficiency follows from Lemmas \ref{cons-pro} and \ref{2-con}, and the necessity follows from Corollary \ref{consmth} and Lemma \ref{conn-cut}.
\end{proof}

\begin{rmk}
The hypergraph constructed as in Construction \ref{cons-det} may not be connected.
However, by Lemma \ref{2-con} at least one component is minimal non-odd-bipartite as the total number of edges is odd.
For example, the following $4$-uniform hypergraph $G$ on $18$ vertices with edges
$$e_t=\{2t-1, 2t, 2t+5,2t+6\}, t \in [9],$$
where the labels of the vertices are modulo $18$.
$G$ has $3$ connected components $G_1,G_2,G_3$ with edge sets listed below, each of which is isomorphic to $C_3^{4,2}$ (a  minimal non-odd-bipartite hypergraph).
$$\begin{array}{clll}
 E(G_1):&  \{1,2,7,8\}, & \{7,8,13,14\}, & \{13,14,1,2\}. \\
 E(G_2): &  \{3,4,9,10\}, & \{9,10,15,16\}, &\{15,16,3,4\}. \\
 E(G_3): &  \{5,6,11,12\}, & \{11,12, 17,0\}, &\{17,0,5,6\}.
\end{array}
$$

In Fig. \ref{illu} we give an illustration of $G$ constructed as in the way of Construction \ref{cons-det},
where the dotted lines indicate a perfect matching in $K_{18,18}$,
and the solid lines indicate a perfect matching in $\hat{K}_{18,18}$.

\begin{figure}
\centering
\setlength{\unitlength}{1bp}%
  \begin{picture}(1061.31, 90.15)(0,0)
  \put(0,0){\includegraphics{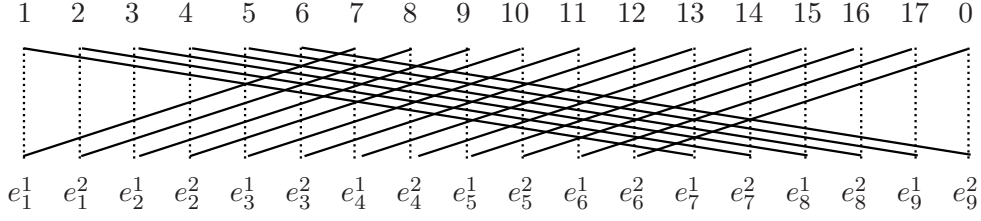}}
    \end{picture}%
\caption{An illustration of Construction \ref{cons-det}}\label{illu}
\end{figure}

\end{rmk}

\subsection{Examples of $d$-regular minimal non-odd-bipartite hypergraphs}

We first give an example of $k$-regular $k$-uniform minimal non-odd-bipartite hypergraph by using Cayley hypergraph.
Let $G=(\Z_n; \{1,2,\ldots,k-1\})$ be a \emph{Cayley hypergraph}, where $V(G)=\Z_n$, and $E(G)$ consists of edges $\{i,i+1,\ldots,i+k\}$ for $i \in \Z_n$.
Then $G$ is connected, $k$-uniform and $k$-regular, with $n$ vertices and $n$ edges.

\begin{thm}\label{d-reg}
Let $k$ be an even integer greater than $2$, and $n$ be an odd integer greater than $k$.
The $G=(\Z_n; \{1,2,\ldots,k-1\})$  is minimal non-odd-bipartite if and only if $\gcd(k,n)=1$.
\end{thm}

\begin{proof}
By Theorem \ref{main}(3), it suffices to show that $\rank B_G=n-1$ over $\Z_2$ if and only if $\gcd(k,n)=1$.
Consider the equation $B_G x=0$ over $\Z_2$.
For each $i \in\Z_n$, as $\{i,\ldots, i+k-1\}$ and $\{i+1,\ldots,k\}$ are edges of $G$, by the above equation we have
$$ x_i+\cdots+x_{i+k-1}=0, x_{i+1}+\cdots+x_{i+k}=0.$$
So $x_i=x_{i+k}$ for each $i \in \Z_n$.
Let $t:=\gcd(k,n)$.
Then there exist integers $p,q$ such that $pk+qn=t$.
Note that $t$ is odd as $n$ is odd, and if writing $k=st$, then $s$ is even as $k$ is even.

For each $i \in \Z_n$,
$$ x_i=x_{i+k}=\cdots=x_{i+pk}=x_{i+t-qn}=x_{i+t}.$$
As $s$ is even, for any $x_1,\ldots,x_t \in \Z_2$, and any edge
 $$x_1+\cdots+x_k=(x_1+\cdots+x_t)+\cdots+(x_{(s-1)t+1}+\cdots+x_{st})=s(x_1+\cdots+x_t)=0.$$
So, the solution space of $B_G x=0$ over $\Z_2$ has dimension $t$, which implies that $\rank B_G=n-t$ over $\Z_2$.
The result now follows.
\end{proof}

Let $G$ be a $k$-uniform hypergraph with $n$ vertices and $m$ edges.
Let $G^1, G^2, \ldots, G^t$ be $t$ disjoint copies of $G$.
For each vertex $v$ (or each edge $e$) of $G$, it has $t$ copies $v^1,\ldots,v^t$ (or $e^1,\ldots,e^t$) in $G^1,\ldots,G^t$ respectively.
Let $t \circ G$ be a hypergraph whose vertex set is $\cup_{i=1}^t V(G^i)$, and edge set is $\{e^1 \cup \cdots \cup e^t: e \in E(G)\}$.
Then $t \circ G$ is $t k$-uniform hypergraph with $t n$ vertices and $m$ edges,
and the degree of $v^i$ in $t \circ G$ is same as the degree of $v$ in $G$ for each $v \in V(G)$ and $i \in [t]$.
If further $G$ is $d$-regular, then $t \circ G$ is also $d$-regular.

\begin{lem}\label{tG}
Let $G$ be a $k$-uniform hypergraph.
Then $G$ is minimal non-odd-bipartite if and only if $t \circ G$ is minimal non-odd-bipartite.
\end{lem}

\begin{proof}
By a suitable labeling of the vertices of $t \circ G$,
we have $B_{t \circ G}=(B_G, B_G, \ldots, B_G)$, where $B_G$ occurs $t$ times in the latter matrix.
As $\rank B_G=\rank B_{t \circ G}$, the result follows by Theorem \ref{main}(3).
\end{proof}

Next we give an example of $d$-regular $k$-uniform minimal non-odd-bipartite hypergraph $G$ with $n$ vertices and $m$ edges,
where $m$ is odd and $d$ is even such that $\gcd(d,m)=1$.
Obviously $nd=km$, and $d|k$ as $\gcd(d,m)=1$.
Suppose $k=t d$, where $t >1$.
By Theorem \ref{d-reg}, the hypergraph $H=(\Z_m; \{1,2,\ldots,d-1\})$ is minimal non-odd-bipartite, which is $d$-regular, $d$-uniform, with $m$ edges.
By Lemma \ref{tG}, $t \circ H$ is minimal non-odd-bipartite with $m$ edges, which is $d$-regular and $td(=k)$-uniform.
%, with $n=mt$ vertices and $m$ edges.

\begin{cor}\label{d-reg-cir}
Let $H=(\Z_m; \{1,2,\ldots,d-1\})$, where $m$ is odd and $d$ is even such that $\gcd(d,m)=1$.
Then $t \circ H$ is minimal non-odd-bipartite with $m$ edges, which is $d$-regular and $t d$-uniform.
\end{cor}

Note that in Corollary \ref{d-reg-cir}, if $d=2$, then $H$ is an odd cycle $C_m$, and $t \circ C_m=C_m^{2t,t}$ (a generalized power hypergraph), both of which are minimal non-odd-bipartite.

Thirdly we use a projective plane $(X,\B)$ of order $q$ to construct a regular minimal non-odd-bipartite hypergraph.
Recall a \emph{projective plane} of order $q$ consists of a set $X$ of $q^2+q+1$ elements called \emph{points},
and a set $\B$ of $(q+1)$-subsets of $X$ called \emph{lines}, such that any two points lie on a unique line.
It can be derived from the definition that any points lies on $q+1$ lines, and two lines meet in a unique point, and there are $q^2+q+1$ lines.
Now define a hypergraph based on $(X,\B)$, denoted by $G=(X,\B)$, whose vertices are the points of $X$ and edges are the lines of $\B$.
Then $G=(X,\B)$ is a $(q+1)$-regular $(q+1)$-uniform  hypergraph with $q^2+q+1$ vertices.

\begin{thm}\label{pp}
Let $(X,\B)$ be a  projective plane of order $q$, and let $G=(X,\B)$ be a hypergraph defined as in the above.
If $q$ is odd, then $G=(X,\B)$ is minimal non-odd-bipartite.
\end{thm}

\begin{proof}
Let $e$ be an edge of $G=(X,\B)$ or a line of $(X,\B)$.
Then
$$B_{G-e}B_{G-e}^\top=qI+J,$$
where $I$ is the identity matrix, and $J$ is an all-ones matrix, both of size $q^2+q$.
So $$\det B_{G-e}B_{G-e}^\top=\det (qI+J)=(q^2+2q)q^{q^2+q-1} \equiv 1 \mod 2,$$
implying that $\rank B_G=m-1$ over $\Z_2$.
The result follows by Theorem \ref{main}(3).
\end{proof}

It is known that if $q$ is an odd prime power, then there always exists a projective plane of order $q$ by using the vector space $\mathbb{F}_q^3$.
By Lemma \ref{tG} and Theorem \ref{pp}, we easily get the following result.

\begin{cor}\label{tpp}
Let $q$ be an odd prime power.
There exists a $(q+1)$-regular $(q+1)$-uniform minimal non-odd-bipartite hypergraph $G$ with $q^2+q+1$ edges.
For any positive integer $t>1$, there exists a $(q+1)$-regular $t(q+1)$-uniform minimal non-odd-bipartite hypergraph with $q^2+q+1$ edges.
\end{cor}

\begin{rmk}
From Corollaries \ref{d-reg} and \ref{tpp}, the minimal non-odd-bipartite hypergraphs $G$ have degree $d$ and edge number $m$ such that $\gcd(d,m)=1$.
(Note that $\gcd(q+1,q^2+q+1)=1$.)
As $\gcd(d,m)=1$, from the equality $nd=mk$, we have $d \mid k$, where $n,k$ are the number of vertices and the uniformity of $G$ respectively.

In fact, there exist $d$-regular minimal non-odd-bipartite hypergraphs with $m$ edges such that $\gcd(d,m)>1$.
For example, let $G$ be a $6$-uniform $6$-regular hypergraph with $9$ edges below:
$$\begin{array}{ccccc}
  \{1,2,3,4,5,6\}, & \{1,4,5,6,7,9\}, & \{1,3,5,6,7,8\}, & \{1,2,4,6,7,8\}, & \{1,3,5,7,8,9\}, \\
  \{1,2,6,7,8,9\}, & \{2,3,4,5,7,9\}, &\{2,3,4,5,8,9\}, &\{2,3,4,6,8,9\}. & ~
\end{array}
$$
By Theorem \ref{main}, it is also easy to verify that $G$ is minimal non-odd-bipartite.

There also exist $d$-regular $k$-uniform minimal non-odd-bipartite hypergraphs such that $d \nmid k$.
For example, let $G$ be a $6$-regular $8$-uniform hypergraph with $9$ edges below:
$$\begin{array}{lll}
  \{1,2,3,4,5,6,7,8\}, & \{1,3,4,5,6,7,9,11\}, & \{1,4,5,6,7,8,9,10\},  \\
  \{1,5,7,8,9,10,11,12\}, & \{1,2,3,6,7,9,10,12\}, & \{1,2,4,5,8,10,11,12\},   \\
   \{2,3,4,6,8,10,11,12\}, & \{2,3,4,5,8,9,11,12\}, & \{2,3,6,7,9,10,11,12\}.
\end{array}
$$
By Theorem \ref{main}, it is easy to verify that $G$ is minimal non-odd-bipartite.
\end{rmk}

\begin{exm}The minimal non-odd-bipartite uniform hypergraphs with fewest edges.
By Theorem \ref{main}, if $G$ is a $k$-uniform minimal non-odd-bipartite hypergraph with $n$ vertices and $m$ edges, then $m$ is odd.
If $m=1$, $G$ is surely odd-bipartite. So, $m \ge 3$, and hence the maximum degree is at most $3$ if $m=3$.
Assume that $m=3$.
By Theorem \ref{main}, each vertex has an even degree, implying that $G$ is $2$-regular.
So $2n=3k$, and $3|n$.
Letting $n=3l$, we have $k=2l$.
So $G=C_3^{2l, l}$, which is the unique example of minimal non-odd-bipartite hypergraph with $3$ edges.
It is consistent with the fact that $C_3$ is the unique minimal non-bipartite simple graph with $3$ edges by taking $k=2$.
\end{exm}

\begin{exm}
The minimal non-odd-bipartite uniform hypergraphs with fewest vertices.
If $G$ is a $k$-uniform minimal non-odd-bipartite hypergraph with $n$ vertices and $m$ edges.
Then $n \ge k+1$, as an edge is odd-bipartite.
Assume that $n=k+1$.
Then $m \le {k+1 \choose k}=k+1$, with equality if and only if $G$ is a \emph{$(k+1)$-simplex} \cite{CD}, i.e. any $k$ vertices of $G$ forms an edge.
Let $\Delta$ be the maximum degree of $G$, which is even by Theorem \ref{main}.
As $m$ is odd by Theorem \ref{main}, we have
$$m-1 \ge \Delta \ge \frac{km}{k+1}=m-\frac{m}{k+1} \ge m-1,$$
which implies that $m=k+1$ and $k$ is even.
So, the $(k+1)$-simplex is the unique example of $k$-uniform minimal non-odd-bipartite hypergraph with $k+1$ vertices by Theorem \ref{main}.
If taking $k=2$, then $C_3$ is the the unique minimal non-bipartite simple graph with $3$ vertices.
\end{exm}

\begin{exm}Example of non-regular minimal non-odd-bipartite hypergraph.
Let $G$ be a $4$-uniform hypergraph on vertices $1,\ldots,9$ with edges
$$\{1,2,3,4\}, \{1,2,4,5\}, \{1,3,6,7\}, \{1,5,8,9\}, \{6,7,8,9\}.$$
It is easy to verify that $G$ is non-regular minimal non-odd-bipartite by Theorem \ref{main}.
\end{exm}

\begin{rmk}
A minimal non-odd-bipartite hypergraph may contain cut edges. For example, the following $4$-uniform hypergraph $G$ with vertex set $[18]$ and $9$ edges:
$$
\begin{array}{lllll}
  \{1,3,4,5\}, & \{2,3,4,6\}, & \{5,7,8,9\}, & \{6,7,8,9\}, & \{1,2,10,11\},\\
  \{10,12,13,14\}, & \{11,12,13,15\}, & \{14,16,17,18\}, & \{15,16,17,18\}, & \\
  \end{array}
$$
where $\{1,2,10,11\}$ is a cut edge of $G$.
By Theorem \ref{main}, $G$ is minimal non-odd-bipartite.
\end{rmk}

\section{Least H-eigenvalue of minimal non-odd-bipartite hypergraphs}
 Let $G$ be a $k$-uniform minimal hypergraph.
 Let $x \in \C^{V(G)}$ whose entries are indexed by the vertices of $G$.
  For a subset $U$ of $V(G)$, denote $x^U:=\Pi_{v \in U} x_u$.
 Then we have
 \begin{equation}\label{form}
  \A(G)x^k = \sum_{e\in E(G)}kx^e,
\end{equation}

\begin{thm}\label{uppb}
Let $G$ be $k$-uniform minimal non-odd-bipartite hypergraph with $n$ vertices and $m$ edges, where $k$ is even.
Then
\begin{enumerate}
  \item $\lamin(G) \le -\rho(G)+\frac{2k}{n^{1/k}}$.

  \item $\lamin(G) \le -(1- {2 \over m})\rho(G)$.
\end{enumerate}
\end{thm}

\begin{proof}
As $G$ is connected by Lemma \ref{conn-cut}, by Perron-Frobenius theorem, there exists a positive eigenvector $x$ of $\A(G)$ associated with the spectral radius $\rho(G)$.
We may assume $\|x\|_k=1$.
Then
\begin{equation}\label{rho}
 \rho(G)=\A(G)x^k=\sum_{e\in E(G)}kx^e.
 \end{equation}
Observe that there exists a vertex $u$ such that $x_u \le \frac{1}{n^{1/k}}$.
Let ${\bar e}$ be an edge of $G$ containing $u$.
Then $$x^{\bar e} = x_u \prod_{v \in {\bar e}, v \ne u}x_v \le \frac{1}{n^{1/k}}.$$
By the definition, $G-{\bar e}$ is odd-bipartite with an odd-bipartition $\{U,U^c\}$.
Now define a vector $y$ on the vertices of $G-{\bar e}$ such that $y_v =x_v$ if $v \in U$ and $y_v =-x_v$ otherwise.
Note that $\bar{e}$ intersects $U^c$ in an even number of vertices as $G$ is non-odd-bipartite, which implies that $y^{\bar{e}}=x^{\bar{e}}>0$.
By Lemma \ref{semi}(2) and Eq. (\ref{rho}),
$$ \lamin(G) \le \A(G)y^k =-\A(G)x^k + 2 k x^{\bar e} \le -\rho(G)+\frac{2k}{n^{1/k}}.$$

For the second result, from Eq. (\ref{rho}), there exists one edge ${\hat e}$ such that $kx^{\hat e}$ is not greater than the average of the summands
$kx^e$ over all $m$ edges $e$ of $G$, that is,
$$kx^{\hat e} \le \frac{\rho(G)}{m}.$$
Note  that $G-{\hat e}$ is also odd-bipartite with an odd-bipartition say $\{W,W^c\}$.
Now define a vector $z$ on the vertices of $G-{\hat e}$ such that $z_v =x_v$ if $v \in W$ and $z_v =-x_v$ otherwise.
By a similar discussion as the above, we have
$$ \lamin(G) \le A(G)z^k =-\A(G)x^k + 2 k x^{\hat e} \le -(1- {2 \over m})\rho(G).$$
\end{proof}

\begin{cor}\label{eps}
Let $k$ be a positive even integer.
For any $\ep>0$, for any $k$-uniform minimal non-odd-bipartite hypergraph $G$ with sufficiently larger number of vertices or edges,
\begin{enumerate}
  \item $-\rho(G)< \lamin(G) < -\rho(G) +\ep$,

  \item $-1< \lamin(G)/ \rho(G) < -1 +\ep$.
\end{enumerate}
\end{cor}

For a connected $k$-uniform hypergraph $G$, where $k$ is even, if we denote 
$$\alpha(G):=\rho(G)+\lamin(G), \;\beta(G):=-\lamin(G)/\rho(G),$$
then by Lemma \ref{ob-equiv}, $\alpha(G) \ge 0$, with equality if $G$ is odd-bipartite; and $0<\beta(G)\le 1$, with right equality if and only if $G$ is odd-bipartite.
So we can use $\alpha(G)$ and $\beta(G)$ to measure the non-odd-bipartiteness of an even uniform hypergraph.

Furthermore, by Theorem \ref{uppb} and Corollary \ref{eps}, 
if $G$ is minimal non-odd-bipartite hypergraph, then $\alpha(G) \to 0$ and $\beta(G) \to 1$ when the number of vertices or edges of $G$ goes to infinity.
So, the minimal non-odd-bipartite hypergraphs  are very close to be odd-bipartite in this sense.

\end{document}